\documentclass[a4paper,12pt]{amsart}

\usepackage{latexsym}
\usepackage{amssymb}
\usepackage{graphicx}
\usepackage{amsthm}
\usepackage{epsfig}
\RequirePackage{color}
\usepackage[normalem]{ulem}

\usepackage{amscd,enumerate}
\usepackage{amsfonts}

\input xy
\xyoption{all}

\usepackage{multicol}
\usepackage{hyperref}
\hypersetup{ colorlinks,
linkcolor=blue,
filecolor=green,
urlcolor=blue,
citecolor=blue }

\newtheorem{lemma}{Lemma}[section]
\newtheorem{corollary}[lemma]{Corollary}
\newtheorem{theorem}[lemma]{Theorem}
\newtheorem{proposition}[lemma]{Proposition}
\newtheorem{remark}[lemma]{Remark}
\newtheorem{definition}[lemma]{Definition}

\newcommand{\supp}{\operatorname{supp}}
\def\a{\alpha}
\def\b{\beta}
\def\Z{\mathbb Z}
\def\N{\mathbb N}
\def\geo{G_E^{(0)}}

\def\te{\mathcal{T}_E}
\def\tec{\mathcal{T}_{E}^c}
\def\tecm{\mathcal{T}_{E}^{cm}}
\def\oe{\mathcal{O}_E}

\def\zz{\mathcal Z}

\newcommand{\stcenter}{{\zz(A_R(G_E))}}

\definecolor{turquoise2}{rgb}{0,0.898039,0.933333}
\definecolor{magenta}{rgb}{1,0,1}

\begin{document}


\author[Lisa O. Clark]{Lisa Orloff Clark}
\address{Lisa O. Clark: Department of Mathematics and Statistics, University of Otago, PO Box 56, Dunedin 9054, New Zealand.}
\email{lclark@maths.otago.ac.nz}

\author[Dolores Mart\'{\i}n]{Dolores Mart\'{\i}n Barquero}
\address{Dolores Mart\'{\i}n Barquero: Departamento de Matem\'atica Aplicada, Escuela T\'ecnica Superior de Ingenieros Industriales, Universidad de M\'alaga. 29071 M\'alaga. Spain.}
\email{dmartin@uma.es}

\author[C\'andido Mart\'{\i}n]{C\'andido Mart\'{\i}n Gonz\'alez}
\address{C\'andido Mart\'{\i}n Gonz\'alez:  Departamento de \'Algebra Geometr\'{\i}a y Topolog\'{\i}a, Fa\-cultad de Ciencias, Universidad de M\'alaga, Campus de Teatinos s/n. 29071 M\'alaga. Spain.}
\email{candido@apncs.cie.uma.es}

\author[Mercedes Siles ]{Mercedes Siles Molina}
\address{Mercedes Siles Molina: Departamento de \'Algebra Geometr\'{\i}a y Topolog\'{\i}a, Fa\-cultad de Ciencias, Universidad de M\'alaga, Campus de Teatinos s/n. 29071 M\'alaga.   Spain.}
\email{msilesm@uma.es}

\thanks{
The first author is supported by the Marsden grant 15-UOO-071 from the Royal Society of New Zealand and a University of Otago Research Grant.
The last three authors are supported by the Junta de Andaluc\'{\i}a and Fondos FEDER, jointly, through projects  FQM-336 and FQM-7156. 
They are also supported by the Spanish Ministerio de Econom\'ia y Competitividad and Fondos FEDER, jointly, through project  MTM2013-41208-P.
\newline
This research took place while the first author was visiting the Universidad de M\'alaga. She thanks her coauthors for their hospitality.}

\title[Using the Steinberg algebra model to determine the center of any LPA]{Using the Steinberg algebra model to determine the center of any Leavitt path algebra}
\begin{abstract} {
Given an arbitrary graph, we describe the center of its Leavitt path algebra over a commutative unital ring. 
Our proof uses the Steinberg algebra model of the Leavitt path algebra. 
A key ingredient is a characterization of compact open 
invariant subsets of the unit space of the graph groupoid in terms of the 
underlying graph: an open invariant subset is compact if and only 
if its associated hereditary and saturated set of vertices satisfies Condition (F).
We also give a basis of the center. Its cardinality depends on the number of minimal compact open invariant subsets
of the unit space.}
\end{abstract}

\maketitle

%%%%%%%%%%%%%%%%%%%%%%%%%%%%%%
%%%%%%%%%%%%%%%%%%%%%%%%%%%%%%
\section{Introduction and preliminary results}
%%%%%%%%%%%%%%%%%%%%%%%%%%%%%%
%%%%%%%%%%%%%%%%%%%%%%%%%%%%%%

In this paper we use the Steinberg algebra model to determine the center of { a } Leavitt path algebra associated to an 
arbitrary graph over a commutative unital ring {$R$}.

The study of the center of certain graph algebras over a field has been 
delivered in a series of 
papers: in \cite{AC} the authors determine the center of a simple Leavitt path algebra; 
in \cite{CMMSS1}, the centers of path algebras,
prime Cohn path algebras and prime Leavitt path algebras are studied; while in \cite{CMMSS2} 
the center of a Leavitt path algebra associated to a row-finite graph is considered.
More generally, in the setting of Leavitt path algebras over an arbitrary commutative ring,
the center of a simple Leavitt path algebra 
is given in \cite{BaH}.

The class of Steinberg algebras was introduced in \cite{Steinberg} to describe discrete inverse semigroup algebras. Each Steinberg algebra is built 
from an ample groupoid  and consists of locally constant, compactly supported functions from the groupoid to the 
coefficient ring. An elegant description of the center of an arbitrary Steinberg algebra is 
given in {\cite[Proposition~4.13]{Steinberg}. Concretely, it consists of the set of \emph{class functions}.}

Steinberg algebras can {also be used to model
Leavitt path algebras, as was first shown in  \cite{CFST} and then generalized in \cite{CS}. To an arbitrary graph $E$, we associate an
ample groupoid $G_E$ as in \cite{Pat2002}.  Then the Steinberg $R$-algebra $A_R(G_E)$ is isomorphic to the Leavitt path algebra $L_R(E)$.}
We use the characterization of the center of the Steinberg algebra  to describe the center of any Leavitt path algebra 
$L_R(E)$ in terms of the underlying graph.

We have divided the paper into three sections. In the first section, we give the preliminaries needed to 
understand the rest of the paper. In the second section,  we characterise 
compactness of open invariant subsets of $\geo$ in terms of the underlying graph.  
To do this, we use the lattice isomorphism established in \cite[Theorem~3.3]{CMMS} between 
open invariant subsets of $\geo$ and pairs $(H,S)$ where $H$ is a hereditary and saturated set 
of vertices and $S$ is a set of breaking vertices associated to $H$.  
We show in Proposition \ref{prop:compacidad}
that an open invariant subset $U_{H, S}$ is compact 
if and only if $H$ satisfies Condition (F), first established in this article.
Roughly speaking, Condition (F) says that $H$ has a finite number of vertices, 
that there is a finite number of paths arriving for the first time 
at $H$ and  not passing through a breaking vertex, and that there are finitely many paths arriving at a 
breaking vertex of $H$.   We are interested in compact open invariant subsets  of $\geo$ 
because they are precisely the supports of the class functions in the 
zero component  of the center.

In Section \ref{TheCenter} we prove the main result of the paper, Theorem \ref{thm:center}, which describes the homogeneous components
of the center of the Leavitt path algebra $L_R(E)$ in terms of the graph. 

The $n$-component of the center of a Leavitt path algebra will be nonzero, for some $n\neq 0$, precisely 
when there exists a cycle $c$ without exits such that { the open invariant subset of $\geo$ that corresponds to 
the saturated hereditary closure of the set of vertices in $c$} is compact.
Finally, Theorem \ref{thm:base} gives a basis for the center of any Leavitt path algebra $L_R(E)$.

\medskip
%%%%%%%%%%%%%%%%%%%%%%%%%%%%%%%%%%%%%%%%%%%

\noindent \textbf{Directed Graphs.}   A \emph{directed graph} is a 4-tuple $E=(E^0, E^1, r_E, s_E)$ 
consisting of two disjoint sets $E^0$, $E^1$ and two maps
$r_E, s_E: E^1 \to E^0$. The elements of $E^0$ are the \emph{vertices} and the elements of 
$E^1$ are the edges of $E$.  Further, for $e\in E^1$, $r_E(e)$ and $s_E(e)$ are 
called the \emph{range} and the \emph{source} of $e$, respectively. 
If there is no confusion with respect to the graph we are considering, we simply write $r(e)$ and $s(e)$.

A vertex {$v$ such that $s^{-1}(v) = \emptyset$ is called a \emph{sink};
$v$ is called an \emph{infinite emitter} if $s^{-1}(v)$ is an infinite set.  Otherwise, a vertex
 that is neither a sink nor an infinite emitter is called a \emph{regular vertex}.}  
The  set of infinite emitters will be denoted by ${\rm Inf}(E)$ while ${\rm Reg}(E)$ will denote the set of regular vertices.

In order to define the Leavitt path algebra, we need to introduce the  
{\it extended graph of} $E$.  This is the graph 
$\widehat{E}=(E^0,E^1\cup (E^1)^*, r_{\widehat{E}}, s_{\widehat{E}}),$ where
$(E^1)^*=\{e_i^* \ | \ e_i\in  E^1\}$ and the functions $r_{\widehat{E}}$ and $s_{\widehat{E}}$ are defined as 
\[{r_{\widehat{E}}}_{|_{E^1}}=r,\ {s_{\widehat{E}}}_{|_{E^1}}=s,\
r_{\widehat{E}}(e_i^*)=s(e_i), \hbox{ and }  s_{\widehat{E}}(e_i^*)=r(e_i).\]
The elements of $E^1$ { are } called \emph{real edges}, while {for each  $e\in E^1$ we} call $e^\ast$ a
\emph{ghost edge}.    

A nontrivial \emph{path} $\mu$ in a graph $E$ is a finite sequence of edges { $\mu=\mu_1\dots \mu_n$
such that $r(\mu_i)=s(\mu_{i+1})$} for $i=1,\dots,n-1$.\footnote{{Some of our references use the opposite path convention where
$s(\mu_i)=r(\mu_{i+1})$.}}
In this case, {$s(\mu):=s(\mu_1)$ and $r(\mu):=r(\mu_n)$} are the
\emph{source} and \emph{range} of $\mu$, respectively, and $n$ is the \emph{length} of $\mu$, denoted $|\mu|$. We also say that
$\mu$ is {\emph{a path from $u:=s(\mu_1)$ to $v:=r(\mu_n)$}} and write $u\geq v$. 
We  { write  $\mu^0$ for} the set of the vertices of $\mu$, i.e.,
$\mu^0:=\{s(e_1),r(e_1),\dots,r(e_n)\}$. 

We view an element $v$ of $E^{0}$ as a path of length $0$. In this case $s(v)=r(v)=v$. 
The set of all (finite) paths of a graph $E$ is denoted by ${\rm Path}(E)$.
{ An \emph{infinite path} is an infinite sequence of edges $x=x_1x_2\dots$ such that
$r(x_i)=s(x_{i+1})$ for every $i\in \N$.  We define the source of an infinite path to be $s(x):=s(x_1)$.}
 We denote the set of all infinite paths by $E^{\infty}$. 

A subset $H$ of $E^0$ is called \emph{hereditary} if $v\ge w$ and $v\in H$ imply $w\in H$. A
hereditary set is \emph{saturated} if every regular vertex which feeds into $H$ and only into $H$ is again
in $H$, that is, if $s^{-1}(v)\neq \emptyset$ is finite and $r(s^{-1}(v))\subseteq H$ imply $v\in H$.  
For a hereditary subset $H$ we denote by $\overline H$ the saturated closure of $H$, i.e., the smallest 
hereditary and saturated subset of $E^0$ containing $H$ (see \cite{AAS} for the construction of $\overline H$).
{Given a non empty hereditary and saturated subset $H$ of $E^0$, define
$$F_E(H):=\{\a \in {\rm Path}(E) \mid  s(\a_1)\in E^0\setminus H, r(\a_i)\in 
E^0\setminus H \mbox{ for } i<\vert \alpha\vert, r(\a_{\vert \alpha\vert})\in H\}.$$ 
}

\medskip
%%%%%%%%%%%%%%%%%%%%%%%%%%%%%%%%%%%%%%%%%%%%%%%%%%%

\noindent \textbf{Leavitt path algebras.}   Given a (directed) graph $E$ and a commutative unital ring  $R$, the {\it path $R$-algebra} of $E$,
denoted by $RE$, is defined as the free associative $R$-algebra generated by the
set of paths of $E$ with relations:
\begin{enumerate}
\item[(V)] $vw= \delta_{v,w}v$ for all $v,w\in E^0$.
\item [(E1)] $s(e)e=er(e)=e$ for all $e\in E^1$.
\end{enumerate}
The {\it Leavitt path algebra of} $E$ {\it with coefficients in} $R$, denoted $L_R(E)$, 
is the quotient of the path algebra $R\widehat{E}$ by the ideal of $R\widehat{E}$ generated by the relations:
\begin{enumerate}    
\item[(CK1)] $e^*e'=\delta _{e,e'}r(e) \ \mbox{ for all } e,e'\in E^1$.
\item[(CK2)] $v=\sum _{\{ e\in E^1\mid s(e)=v \}}ee^* \ \ \mbox{ for every}\ \ v\in  {\rm Reg}(E).$
\end{enumerate}
Observe that in $R\widehat{E}$ the relations (V) and (E1) remain valid and that the following is also satisfied:
\begin{enumerate}
\item [(E2)] $r(e)e^*=e^*s(e)=e^*$ for all $e\in E^1$.
\end{enumerate}
It is not difficult to show that
\[L_R(E) = \operatorname{span} \{\alpha\beta^{\ast } \ \vert \ \alpha, \beta \in {\rm Path}(E)\}.\]
and that $L_{R}(E)$ is a $\mathbb{Z}$-graded $R$-algebra, where
for each $n\in\mathbb{Z}$, the degree $n$-component $L_{R}(E)_{n}$ is spanned by the set
\[\{\alpha \beta^{\ast }\ \vert \  \alpha, \beta \in {\rm Path}(E)\ \hbox{and}\  |\alpha|-|\beta|=n\}.\]

\begin{remark} 
\rm Our notation for elements of $L_R(E)$ and elements of 
the set ${\rm Path}(E)$ is the same.  Note that we will often be considering elements of ${\rm Path}(E)$
where we will not assume any of the structure that comes with elements of the quotient 
$L_R(E)$.
\end{remark}

\medskip
%%%%%%%%%%%%%%%%%%%%%%%%%%%%%%%%%%%%%%%%%%%%%%%%%%%%

\noindent\textbf{Groupoids and Steinberg algebras.}  
A groupoid $G$ is a generalisation of a group in which the `binary operation' is only partially defined. For a 
precise definition see \cite{Renault}.  The unit space of $G$ is denoted $G^{(0)}$  and is defined such that
\[G^{(0)} := \{\gamma \gamma^{-1} \ \vert \ \gamma \in G\} = \{\gamma^{-1}\gamma \ \vert \ \gamma \in G\}.\]
Groupoid \emph{source} and \emph{range} maps $r,s:G \to G^{(0)}$ are defined such that
\[
 s(\gamma) = \gamma^{-1} \gamma \quad \text{ and } \quad r(\gamma) = \gamma \gamma^{-1}.
\]
Although the notation and terminology of $r$ and $s$ have both a graphical and a groupoid interpretation,
it will be clear from context which one we mean.
{We say $U \subseteq G^{(0)}$ is \emph{invariant} if $s(\gamma) \in U$ implies 
$r(\gamma) \in U$.  Equivalently, $U$ is invariant if $r(\gamma) \in U$ implies $s(\gamma) \in U$.}

Steinberg algebras, introduced in \cite{Steinberg}, are algebras associated to \emph{ample} groupoids and 
are defined 
as follows.   First, a \emph{topological groupoid} is a groupoid equipped with a topology such that
composition and inversion are continuous.  Next, we say an open subset $B$ of a groupoid $G$ is 
an \emph{open bisection}
if $r$ and $s$ restricted to $B$ are homeomorphisms onto an open subset of $G^{(0)}$.  
Finally, we say a topological groupoid $G$ is \emph{ample} if there is a basis for its topology 
consisting of compact open bisections.

{
Suppose $G$ is a Hausdorff ample groupoid and $R$ is a commutative unital ring.
The Steinberg algebra associated to $G$ denoted $A_R(G)$ is 
the $R$-algebra  of all locally constant functions $f:G \to R$ such that
\[\supp f =\{\gamma \in G \ \vert \ f(\gamma) \neq 0\}\] is compact.  
Thus, for $f \in A_R(G)$, we have that $\supp f$ is both compact and open.  }
It turns out that
\[A_R(G) = \operatorname{span} \{1_B \ \vert \ B \text{ is a compact open bisection}\},\] 
where $1_B:G \to R$ is the characteristic function of $B$.  Addition and scalar multiplication 
in $A_R(G)$ are defined pointwise.  Multiplication is given by convolution such that
for compact open bisections $B$ and $D$ we have $1_B 1_D = 1_{BD}.$ 

\medskip

{The center of the Steinberg algebra $A_R(G)$, denoted $\zz(A_R(G))$, is characterised in \cite[Proposition~4.13]{Steinberg}.
First, we say $f \in A_R(G)$ is a \emph{class function} if $f$ satisfies the following conditions:
\begin{enumerate}
 \item\label{it1:class} $f(x)\neq 0 \implies s(x)=r(x);$ 
 \item\label{it2:class} $s(x)=r(x)=s(z) \implies f(zxz^{-1})=f(x).$ 
\end{enumerate}
Then  \cite[Proposition~4.13]{Steinberg} says that the center of $A_R(G)$ is
\[\zz(A_R(G)) = \{f \in A_R(G) \ \vert \ f \text{ is a class function}\}.\]}

\noindent{
\noindent\textbf{The Steinberg algebra model of a Leavitt path algebra.}
Next, we recall the construction of a groupoid $G_E$ from an arbitrary graph $E$,
which was introduced in \cite{KPRR} for row-finite graphs 
and generalised to arbitrary graphs in \cite{Pat2002}.  We use the notation of 
\cite[Example~2.1]{CS}.} 
Define
\begin{align*}X &:= E^\infty \cup \{\mu \in {\rm Path}(E) \mid r(\mu)\text{ is a sink}\} \cup 
	\{\mu \in {\rm Path}(E) \ \vert \ r(\mu) \in \operatorname{Inf}(E)\} \text{ and }\\
G_E &:= \{(\alpha x, |\alpha| - |\beta|, \beta x) \ \vert \ \alpha,\beta \in {\rm Path}(E),
    x \in X, r(\alpha) = r(\beta) = s(x)\}.\end{align*}
A pair of elements in $G_E$ is {\it{composable}} if and only if it is of the form $((x,k,y),(y,l,z))$ and then
the composition and inverse maps are defined such that
\[
    (x,k,y)(y,l,z) := (x, k+l, z) \quad\text{ and }\quad (x,k,y)^{-1} := (y,-k,x).
\]
Thus
\[\geo =\{(x, 0, x) \ \vert \ x \in X\}, \text{ which we identify with $X$}.\]

Next we show how $G_E$ can be viewed as an ample groupoid. For $\mu \in {\operatorname{Path}}(E)$ define
\[
    Z(\mu):= \{\mu x \mid x \in X, r(\mu)=s(x)\} \subseteq X.
\]    
    
 For $\mu \in {\rm Path}(E)$ and a finite $F \subseteq s^{-1}(r(\mu))$, define
\[
Z(\mu\setminus F) := Z(\mu) \cap \Big(\geo \setminus\Big(\bigcup_{\alpha \in F} Z(\mu\alpha)\Big)\Big).
\]

Then the sets of the form $Z(\mu \setminus F)$ are a basis of compact open sets for a 
Hausdorff topology on $X = \geo$ by \cite[Theorem~2.1]{Webster:xx11}.

\medskip

For each $\mu,\nu \in {\rm Path}(E)$ with $r(\mu) = r(\nu)$, 
and finite $F \subseteq {\rm Path}(E)$ such that $r(\mu)=s(\alpha)$ 
for all $\alpha \in F$, define
\[
Z(\mu,\nu) := \{(\mu x, |\mu| - |\nu|, \nu x) \ \vert \ x \in X, r(\mu)= s(x)\},
\]
and then {
\[
Z((\mu,\nu) \setminus F) := Z(\mu,\nu) \cap
            \big(G_E \setminus \Big(\bigcup_{\alpha \in F} Z(\mu\alpha,\nu\alpha)\Big)\Big).
\]}
Now the collection $Z((\mu,\nu) \setminus F)$ forms a basis of 
compact open bisections that generates a topology such that
$G_E$ is a Hausdorff ample groupoid.  
\bigskip

{Now \cite[Example 3.2]{CS} shows that the map
\begin{equation}
 \label{def:pi} 
 \pi:L_R(E) \to A_R(G_E) \text{ such that }  \pi(\mu\nu^\ast-\sum_{\alpha\in F}\mu\a\a^\ast\nu^\ast) =1_{Z((\mu, \nu)\setminus F)}
\end{equation}
extends to a $\Z$-graded isomorphism where for $n \in \Z$
\[A_R(G_E)_n := \{f \in A_R(G_E) \ \vert \ (x,k,y) \in \supp f \implies k=n\}. 
\]}

\bigskip

%%%%%%%%%%%%%%%%%%%%%%%%%%%%%%%%%%%%%%%%%%%5

%%%%%%%%%%%%%%%%%%%%%%%%%%%%%%
%%%%%%%%%%%%%%%%%%%%%%%%%%%%%%
\section{Compactness}

%%%%%%%%%%%%%%%%%%%%%%%%%%%%%%
%%%%%%%%%%%%%%%%%%%%%%%%%%%%%%

In this section we establish a key tool for determining the center of a Leavitt path algebra $L_R(E)$ via the Steinberg algebra model. 
We characterize when an open invariant subset of $\geo$ is also compact in terms of $E$.

Let $H$ be a hereditary and saturated subset  of $E^0$, and let $v\in E^0$.  We say that $v$ is a 
\index{breaking vertex} \emph{breaking vertex of} $H$ if $v$ belongs
to the set
$$\index{$B_H$} B{_H}:=\{v\in E^0\setminus H \ \vert \ v\in {\rm Inf}(E) \ \text{and} \ 0 < \vert s^{-1}(v) \cap r^{-1}(E^0\setminus H)\vert < \infty\}.$$
(See, for example, \cite[Definitions 2.4.3]{AAS}).
In words, $B_H$ consists of those vertices which are infinite emitters, which do not belong to $H$, 
and for which the ranges of the edges they emit are all, except for a finite (but nonzero) number, inside $H$.
For  $\a\in {\rm Path}(E)$  such that $r(\a)\in B_H$ write
\[F_\a:= \{e\in s^{-1}(r(\a))\ \vert \ r(e)\notin H \} \quad \text{and} \quad F'_\a:= \{e\in s^{-1}(r(\a))\ \vert \ r(e)\in H \}.\]
Thus $F_\a$ is a finite set, $F'_\a$ is an infinite set and $s^{-1}(r(\a)) = F_\a \cup F'_\a$.

Define, as in \cite{AAS}
 \[\te:= \{(H,S) \ \vert \  H \text{ is a  
hereditary and saturated set of vertices and } S \subseteq B_H\}.\]
Following \cite[Definitions 3.3]{CMMS} for $(H,S) \in \te$ we
define
\begin{align}
 U_{H}&:=\{x\in\geo \ \vert\ r(x_n)\in H \text{ for some } n\geq 0\},\label{UH}\\
 U_S& :=\{\alpha\in {\rm Path}(E)\ \vert\ r(\alpha)\in S \} \label{US} \text{ and}\\
 U_{H,S}&:=U_H\sqcup U_S. \label{UHS}
\end{align}
(Note for a path $x$, we adopt the convention that $r(x_0):= s(x)$.)
Then the map $(H,S) \mapsto U_{H,S}$ from $\te$ to the collection $\oe$ of open
invariant subset of $\geo$ is a bijection by \cite[Theorem~3.3]{CMMS}.
In the next lemma, we show how each $U_{H,S}$ can be expressed as a disjoint union
of basis elements.  First, for a hereditary and saturated set of vertices $H$, define
$$F_E'(H):=\{\a \in F_E(H)\ \text { such that } \ s(\a_{\vert\a\vert})\notin B_H  \}.$$

\begin{lemma}\label{lm:disjoint}
Let $E$ be an arbitrary graph and $(H, S)\in \te$. 
Then:
$$
U_{H, S} = \Big(\underset{v\in H}\sqcup{Z(v)}\Big)\sqcup
\Big( \underset{\a\in F'_E(H)}\bigsqcup Z(\a)\Big) \sqcup 
 \Big( {\underset{
\tiny{\begin{matrix}
\a\in {\rm Path}(E)\\
r(\a) \in S
\end{matrix}}}{\bigsqcup}} Z(\a\setminus F_\a)\Big).
$$
\end{lemma}
\begin{proof}
Denote by 
$A_1$, $A_2$ and $A_3$, respectively, the three pairwise disjoint sets appearing in the 
last line of the statement of the lemma.
First we show that $U_{H, S} \subseteq A_1\cup A_2\cup A_3$.
Notice that $U_S \subseteq A_3$. Pick $x\in U_H$.  Then  there exists 
$n$ such that $r(x_n)\in H$.
If $x$ is a sink or an infinite emitter, then it is in $H$ and we are done. 
If $x$ is not a vertex, fix the smallest $n \geq 0$ such that $r(x_n) \in H$ and let \[
\alpha = \begin{cases} x_1\dots x_{n} & \text{ if }n>0 \text{ and }\\
          s(x) &\text{ otherwise}.
         \end{cases}\]
Then $\alpha \in F_E(H)\cup H$ and   
$x = \alpha x_{n}x_{n+1} \dots \in Z(\alpha)$ or $x\in Z(\a\setminus F_\a)$ and we are done.
{Thus we have} $U_{H, S}\subseteq A_1 \cup A_2 \cup A_3$.

{The reverse containment and that the sets involved are disjoint are both straightforward.}
\end{proof}

\begin{definition}
\rm
Let $H$ be a hereditary and saturated set of vertices in an arbitrary graph $E$. We say that $H$ satisfies the \emph{Finiteness Condition} 
(\emph{Condition} (F) for short) if 
$$\vert H \cup F_E'(H)\ \cup \{\a\in {\rm Path}(E) \ \vert \ r(\a) \in B_H\}\vert < \infty.
$$
\end{definition}

\begin{proposition}\label{prop:compacidad}
Let $E$ be an arbitrary graph. Fix $(H, S)\in \te$. Then $U_{H, S}$ is compact if and only if 
$H$ satisfies Condition {\rm (F)}.
\end{proposition}
\begin{proof}
Suppose $U_{H, S}$ is compact.  By way of contradiction, suppose that {
\[\vert H \cup F_E'(H)\ \cup \{\a\in {\rm Path}(E) \ \vert \ r(\a) \in B_H\}\vert =\infty.\]}   
Since  the sets involved in Lemma \ref{lm:disjoint} provide a disjoint open cover of $U_{H, S}$ with no finite 
subcover  we get a contradiction to the fact that $U_{H, S}$ is compact and therefore Condition (F) is satisfied.

For the reverse implication, suppose that $H$ satisfies Condition (F).  
By Lemma \ref{lm:disjoint}  we have that $U_{H,S}$ is a union of a finite number of compact sets, thus $U_{H,S}$ is compact.
\end{proof}
\medskip

\begin{remark}\label{rmk:closed}
\rm
Looking at the graph properties involved in Proposition \ref{prop:compacidad} and in \cite[Proposition 4.2]{CMMS} we have another proof of the fact that  if 
$U_{H, S}$ is compact, then it is closed.
On the other hand, by \cite[Remark 4.2]{CMMS}, { $U_{H, S}$ being } closed implies 
$S=B_H$. In fact, if $H$ satisfies a weaker property, concretely \cite[Proposition 4.1 (ii)]{CMMS}, 
then $S=B_H$.
\end{remark}

{Let $\tec$ denote the set of elements $(H,S) \in \te$ such that
$S=B_H$ and $U_{H, B_H}$ is compact (that is,  $H$ satisfies Condition (F) by Proposition~\ref{prop:compacidad}).
Further, let $\tecm$ denote the subset of $\tec$ consisting of those pairs $(H, B_H)$ such that $U_{H, B_H}$ is 
minimal as a compact set.}
\medskip

%%%%%%%%%%%%%%%%%%%%%%%%%%%%%%%%%%%%
{Before moving from the graph and groupoid structure into the structure of the algebras, 
we investigate how cycles without exit give rise to elements of $\tec$. 
The following characterization will help us
to show which cycles without exit contribute to the $n$-component of the center of a Leavitt path algebra.}

\begin{lemma}\label{lem:cicloscentrales}
Let $c$ be a cycle without exits in an arbitrary graph $E$. Let $H:=\overline{c^0}$. The following are equivalent:
\begin{enumerate}[\rm (i)]
\item $\vert F_E(c^0) \vert < \infty$.
\item $H$ satisfies Condition {\rm (F)} and $B_H=\emptyset$.
\end{enumerate}
\end{lemma}
\begin{proof}
(i) $\Rightarrow$ (ii). Since $\vert F_E(c^0) \vert < \infty$, necessarily  $B_H=\emptyset$. This implies $F'_E(H)=F_E(H)$, hence, to show that $H$ satisfies Condition (F) it suffices to see that $\vert H \vert< \infty$. 
Assume on the contrary that $\vert H \vert = \infty$. Let $\{u_i\}_{i \in \N}\subseteq H \setminus c^0$ be a subset of different elements. By the definition of the saturated closure of a set, for any $i \in \N$ there exists a path $\mu_i\in F_E(c^0)$ such that $s(\mu_i)=u_i$. 
Then $\{\mu_i\}_{i\in \N}$ is an infinite set of paths inside $F_E(c^0)$, a contradiction. 

(ii) $\Rightarrow$ (i).  The hypothesis in (ii) imply $\vert H \cup F_E(H)\vert$. To show our claim, let $\alpha \in F_E(c^0)$. We may decompose $\alpha=\beta\gamma$, where $\beta$ and $\gamma$ are paths such that $\beta\in F_E(H)$. Since $F_E(H)$ is finite, to show our claim it is enough to prove that there are finitely many paths in $F_E(c^0)$ whose source is in $H$. Note that this is true because $H$ is finite and, by the construction of the saturated closure,  $H=\overline{c^0}$
contains the vertices in $c^0$ and other regular vertices which are not the { base} of any cycle.
\end{proof}

%%%%%%%%%%%%%%%%%%%%%%%%%%%%%%
%%%%%%%%%%%%%%%%%%%%%%%%%%%%%%
\section{The center}\label{TheCenter}

%%%%%%%%%%%%%%%%%%%%%%%%%%%%%%
%%%%%%%%%%%%%%%%%%%%%%%%%%%%%%

In this section we will prove the main result of the paper, which is 
Theorem \ref{thm:center}.  It determines the elements in the center of any Leavitt path algebra
$L_R(E)$. We first describe those central elements in the zero component and then {we consider } 
those in the $n$-component (which {is nonzero }
only when there are cycles without exit). Moreover, we determine a basis for the center of $L_R(E)$.

First we establish some { additional } notation.
Let $E$ be an arbitrary graph and let $c=e_1\dots e_n$ be a cycle based at $v=s(e_1)$. Then, for each $2\leq i \leq n$ the path 
\[c_i= e_i\dots e_ne_1\dots e_{i-1}\]
is a cycle based at the vertex $s(e_i).$
We will refer to \emph{the cycle of} {$[c]$ as the collection of cycles $\{c_i\}_{i=1}^n$, where $c_1:= c$.}
Denote by ${\mathcal C}_{ne}$ the set of all cycles without exits and by 
${\mathcal C}_{ne}^f $ the subset of ${\mathcal C}_{ne}$ consisting of those  {cycles  $c$ 
for which the number of paths into $c$ is finite; that is $\vert F_E({c^0}) \vert < \infty$. 
Finally, } define 
$$[{\mathcal C}_{ne}^f ]: =\{[c] \ \vert \ c\in {\mathcal C}_{ne}^f \}.$$
\medskip

{Recall that the center of $A_R(G_E)$ is the set of all class functions by \cite[Proposition~4.13]{Steinberg}.  
Before proving our main theorem, we establish some results about the structure of class functions.}  
%%%%added by Can %%%%%%%
\begin{lemma}\label{UNO}
Let $E$ be an arbitrary graph, $R$ a  commutative unital ring and let $U$ be a compact open invariant subset of $\geo$. Then $1_U$ is a class function. 

\end{lemma}
\begin{proof}
Assume that $U$ is a compact open invariant subset. It is straightforward to prove 
that $1_U :G_E\to R$ is a locally constant function (whose support is $U$). 
Further, Condition \eqref{it1:class} in the definition of a class function 
is satisfied because $U\subseteq \geo$ .  Now we prove Condition \eqref{it2:class}.
Let $x\in U$ and let $z=(y, k, x)\in G_E$. Then $zxz^{-1}= (y, 0, y) $. 
Since $U$ is invariant, $y\in U$. Similarly,  given that $\geo \setminus U$ is invariant, we have that  
$y\in \geo\setminus U$ for any $x\notin U$ and $z=(y, k, x)\in G_E$. This proves Condition \eqref{it2:class}.
\end{proof}
\medskip

\begin{lemma} \label{lm:coisupport}   
Let $E$ be an arbitrary graph and let $R$ be a commutative unital ring. Take $f\in A_R(G_E)$.  Then:
\begin{enumerate}[\rm (i)]
\item\label{it1:coisupport} $f^{-1}(b)$ is a compact open set for all nonzero $b \in R$.
 \end{enumerate}

  Furthermore, if $f \in \stcenter_0$ we have:
 
 \begin{enumerate}[\rm (i)]
 \setcounter{enumi}{1}
 \item\label{it2:coisupport} $\supp(f)\subseteq \geo$\textcolor{blue}{;}
  \item\label{it3:coisupport} {$f^{-1}(b)$} is a compact open invariant set for all nonzero $b \in R$; and
 \item\label{it4:coisupport} $\supp(f)$ is a compact open invariant set. 
\end{enumerate}
\end{lemma}

\begin{proof}
\begin{enumerate}[\rm (i)]
 \item[(\ref{it1:coisupport})]  Because $f$ is continuous with respect to the discrete topology on $R$,  
  $f^{-1}(b)$ is clopen.  Since $\supp f$ is compact and $f^{-1}(b)$ is a closed subset of $\supp f$,
  $f^{-1}(b)$ is compact.  
 \item[(\ref{it2:coisupport})] Since $f$ is a class function in the 0-component,
 $${\rm supp}(f)\subseteq \{(x, k, y)\in G_E \ \vert \ x= y \ \text{ and }\ k=0\}=\geo.$$
 \item[(\ref{it3:coisupport})] By \eqref{it1:coisupport}, $f^{-1}(b)$ is compact and open.  We want to show that $f^{-1}(b)$ is invariant. 
Fix $y\in f^{-1}(b)$. Suppose there exist $x\in \geo$ and $k\in \mathbb{Z}$ such that $(x, k, y)\in G_E$. 
Since $(y, 0, y )= (y, -k, x)(x, 0, x) (x, k, y)$ and $f$ is a class function, $f(y)=f(x)=b$.  This means $x\in f^{-1}(b)$.
  
  \item[(\ref{it4:coisupport})] Since $\supp f$ is always compact and open for elements of $A_R(G_E)$, we
  must only show that $\supp f$ is invariant.  This follows from item~(\ref{it3:coisupport}) because
  \[\supp f = \bigcup_{b\in R,~b\neq 0} f^{-1}(b)\]
and  the union of invariant sets is invariant.
 \end{enumerate}

\end{proof}

{In the following two lemmas we show that cycles without exit give rise to class functions
that have support outside of $\geo$.} 
Extending the notation used for finite paths, 
for a (finite or infinite path) $x$, we denote by $x^0$ the set of vertices which are ranges 
or sources of the edges appearing in $x$.  {In the case where $x$ is} a vertex, $x^0=\{x\}$. 

\begin{lemma}\label{lem.primero}
Let $E$ be an arbitrary graph, $\alpha, \beta\in {\rm Path}(E)$ with $r(\alpha) = r(\beta)$ and $r(\a)E^\infty\ne\emptyset$. 
Then the following conditions are equivalent:  
\begin{enumerate}[\rm (i)]
\item\label{qwe1} 
For any $x$ in $r(\alpha)E^\infty$ we have $\alpha x = \beta x$.
\item\label{qwe2} {Either }$\a=\b$ or there is a cycle without exits $c$ and a path $\lambda\in {\rm Path}(E)$ such that $\a=\lambda c^k$ and $\b=\lambda c^q$ for some nonnegative 
integers $k,q$. 
\end{enumerate}
\end{lemma}

\begin{proof}
 (\ref{qwe2}) $\Rightarrow$ (\ref{qwe1}) is easy.
To prove (\ref{qwe1}) $\Rightarrow$ (\ref{qwe2})
we use induction on the number $n=\vert\a\vert+\vert\b\vert$. If $n=0$ then $\a$ and $\b$ are vertices and $\a=r(\a)=r(\b)=\b$.
Assume now that (\ref{qwe1}) $\Rightarrow$ (\ref{qwe2}) is true whenever $n<k$ and consider $\a$ and $\b$ such that $\vert\a\vert+\vert\b\vert=k$.
If $\vert\a\vert=\vert\b\vert$, take $x\in r(\a)E^\infty$ (which, by hypothesis, is nonempty) and then $\a x=\b x$, which implies $\a=\b$.

Now, assume $\vert \alpha \vert \neq \vert \beta \vert$. Without loss of generality, 
suppose $\vert \alpha \vert > \vert \beta \vert$. Since for all $x \in r(\alpha)E^\infty$ we have $\alpha x = \beta x$, we may write 
$\alpha = \beta \mu$ for some $\mu \in {\rm Path}(E)$. Then $\alpha x = \beta \mu x= \beta x$ for all $x \in r(\alpha)E^\infty$. 
Therefore $\mu x =x$ and we can apply the induction hypothesis. Thus, either $\mu= r(\mu)$, what implies $\a=\b$, or there is a  finite path $\lambda$ and
a cycle without exits $c$,
such that $\mu=\lambda c^k$ and $r(\mu)=\lambda c^q$. Of course, this implies $q=0$ and $\lambda=r(\mu)$. So $\mu=c^k$ and $\alpha=\beta c^k$.
Since $\beta=\beta c^0$ we are done. It may happen that we can \lq\lq extract\rq\rq\ more cycles $c$ from the path $\b$ so that
$\beta=\lambda c^n$ for some $n\ge 0$. If this is the case we may write $\alpha=\lambda c^{n+k}$, $\beta=\lambda c^n$.
\end{proof}

For any cycle without exists $c$ in a graph $E$, define 

\begin{equation}\label{delta}
\Delta : = \{ (\a c^\infty, m \vert c \vert, \a c^\infty) \ \vert \ \a \in F_E({{c^0}}),\  m \in \N\}.
\end{equation}

\begin{lemma}\label{lem.segundo}
Let $E$ be an arbitrary graph, $R$ a unital commutative ring, and  {$f \in \zz(A_R(G_E))$}. Then $\supp f\subseteq G_E^{(0)}\cup \Delta$.
\end{lemma}
\begin{proof}
Fix $f\in Z(A_R(G_E))$. Take $\gamma\in \supp f$. Then, $f(\gamma) \neq 0$ and  
$s(\gamma) = r(\gamma)$ since $f$ is a class function. Because $f$ is locally constant, there exist $\alpha, \beta \in {\rm Path}(E)$ such that $r(\alpha)=r(\beta)$ and $\gamma \in Z(\alpha, \beta)$ and $f(\eta)=f(\gamma)\neq 0$ for all $\eta \in  Z(\alpha, \beta)$. Again, using that $f$ is a class function, $s(\eta)=r(\eta)$ for all $\eta \in  Z(\alpha, \beta)$.
Now we have two possibilities:
\begin{enumerate}
\item If $\alpha$ does not connect to an infinite path, then $\alpha = \beta$.  
Thus, $\gamma \in Z(\alpha, \beta)\subseteq G_E^{(0)}$.
\item If $\alpha$ connects to an infinite path, then  by Lemma~\ref{lem.primero}, either  $\a=\b$, 
or $\a=\lambda c^k$ and $\beta=\lambda c^q$ for some $\lambda\in {\rm Path}(E)$ and some cycle without exits $c$.  In this second case,
$Z(\a,\b)=\{(\lambda c^\infty, (k-q)\vert c\vert,\lambda c^\infty)\}$ and so
$\gamma=(\lambda c^\infty, (k-q)\vert c\vert,\lambda c^\infty)$.
\end{enumerate}
\end{proof}

{Finally, we give one last technical lemma and then we will be ready for our main theorem.}

\begin{lemma}\label{CINCO}
Let $E$ be an arbitrary graph and $R$ a commutative  unital ring. 
If  $f$ is a nonzero element in $A_R(G_E)$, then there exists 
$\{b_i\}_{i=1}^l\subseteq R\setminus \{0\}$ such that $f=\sum_{i=1}^lb_if_i$ where $f_i=1_{f^{-1}(b_i)}\in A_R(G_E)$ {for each $i$}.
Moreover, if $f$ is a class function, then each $f_i$ is also a class function.
\end{lemma}
\begin{proof}
Fix $f\in A_R(G_E)$. Since $f$ is nonzero, locally constant and ${\rm supp} f$ is 
compact and open, there exists $\{b_i\}_{i=1}^m\subseteq R\setminus \{0\}$ such that ${\rm supp} f = \cup_{i=1}^l f^{-1}(b_i)$. 
Therefore $f=\sum_{i=1}^lb_if_i$. 

Now, assume that $f$ is a class function. {To see that each $f_i$ is a class function, notice that
Condition (1) follows from the fact ${\rm supp } f_i \subseteq {\rm supp } f$. 
To show Condition (2)}, consider elements $x, z$ such that $s(x)=r(x)=s(z)$. 
Notice that $x\in f^{-1}(b_i)$ if and only if $zxz^{-1}\in f^{-1}(b_i)$ as follows: $b_i=b_if_i(x) =f(x)= f(zxz^{-1})$. Therefore
$f_i(zxz^{-1}) = f_i(x)$.
\end{proof}
\medskip

\begin{theorem}\label{thm:center}
Let $E$ be an arbitrary graph and $R$ a  commutative unital ring. Then: 
\begin{enumerate}[\rm (i)]
\item The zero component of the center of $L_R(E)$ is 
\begin{equation}\label{center1}
\footnotesize{\zz(L_R(E))_0=
 {\rm span}\Big{\{}\sum_{v\in H}v + \sum_{\a\in F_E'(H)}\a\a^\ast + 
{\underset{
\tiny{\begin{matrix}
\a\in {\rm Path}(E)\\
r(\a) \in B_H
\end{matrix}}}{\sum}} ( \a\a^\ast - \sum_{e\in F_\a}\a e e^\ast \a^\ast) \ \vert \
(H, B_H)\in\tec
 \Big{\}}.}
 \end{equation}
 \item When the $n$-component of the center of $L_R(E)$ is nonzero,  for some $n\neq 0$, it coincides with
 \begin{equation}\label{center2}
  \footnotesize{
\zz(L_R(E))_n = \bigoplus_{\tiny{
 \begin{matrix}
 [c]\ \in\ [{\mathcal C}_{ne}^f] \\ m \vert c \vert = n
 \end{matrix}}
 } {\rm span}\Big{\{} 
 \sum_{
\tiny
{\begin{matrix}
 d \in [c]\\
 \a \in F_E(c^0) \cup \{s(d)\}\\
 r(\a)=s(d)
 \end{matrix}
 }
 }\a d^m\a^\ast
\Big{\}}.}
\end{equation}
\end{enumerate}
\end{theorem}
\begin{proof}
(i). First we show that the {right-hand side of}  
\eqref{center1} is contained in $\zz(L_R(E))_0$. Note that this set is in the zero component of $L_R(E)$. Fix $(H, B_H)\in \tec$ and take 
$$a=\sum_{v\in H}v + \sum_{\a\in F_E'(H)}\a\a^\ast + 
{\underset{
\tiny{\begin{matrix}
\a\in {\rm Path}(E)\\
r(\a) \in B_H
\end{matrix}}}{\sum}} ( \a\a^\ast - \sum_{e\in F_\a}\a e e^\ast \a^\ast).$$
By the definition of {the isomorphism $\pi:L_R(E) \to A_R(G_E)$ given in \eqref{def:pi}} , 
we have $ \pi(a)=1_{U_{H, B_H}}$
{by Lemma~\ref{lm:disjoint}. 
Recall that $(H,B_H) \in \tec$ implies ${U_{H, B_H}}$  is a compact open invariant subset of $\geo$. Now, apply 
Lemma \ref{UNO} to get $ \pi(a)=1_{U_{H, B_H}}\in \zz(A_R(G_E))_0$.}

Next we show that $\zz(L_R(E))_0$ is contained in {the right-hand side of \eqref{center1}. 
Fix $f\in \pi (\zz(L_R(E)))_0$ such that $f = \pi(a)$.  We show $a$
is in the right-hand side of \eqref{center1}. First note that $f$ is a class function by \cite[Proposition 4.13]{Steinberg}. 
Thus, by Lemma \ref{CINCO} 
there exists $\{b_i\}_{i=1}^l\subseteq R\setminus \{0\}$ such that 
$f=\sum_{i=1}^l b_i f_i$, where $f_i=1_{f^{-1}(b_i)}$ and each $f_i$ is a class function. 
Then $f^{-1}(b_i)$ is a compact open invariant subset of $\geo$ by 
Lemma~\ref{lm:coisupport}\eqref{it2:coisupport} and \eqref{it3:coisupport}.
Thus we can write $$f= \sum_{i=1}^l b_i 1_{U_{H_i, B_{H_i}}}$$ where each $(H_i, B_{H_i})\in \te$
by \cite[Theorem 3.4]{CMMS}. In fact each $(H_i, B_{H_i})\in \tec$.
For each $i$ write 
$$1_{U_{H_i, B_{H_i}}}= \sum_{v\in H_i}1_{Z(v)} + \sum_{\a\in F_E'(H_i)}1_{Z(\a)} + 
{\underset{
\tiny{\begin{matrix}
\a\in {\rm Path}(E)\\
r(\a) \in B_{H_i}
\end{matrix}}}{\sum}}1_{Z(\a\setminus F_\a)},$$
using Lemma \ref{lm:disjoint}.  Note that every sum in the expression above is finite and $(H_i, B_{H_i})$ 
satisfies Condition (F) by Proposition~\ref{prop:compacidad}.
So $1_{U_{H_i, B_{H_i}}}=\pi(a_i)$ where 
$$a_i= \sum_{v\in H_i}v + \sum_{\a\in F_E'(H_i)}\a\a^\ast + 
{\underset{
\tiny{\begin{matrix}
\a\in {\rm Path}(E)\\
r(\a) \in B_{H_i}
\end{matrix}}}{\sum}} ( \a\a^\ast - \sum_{e\in F_\a}\a e e^\ast \a^\ast).$$  
Thus each $a_i$ is an element of the right-hand side of \eqref{center1}.  Notice that 
$a = \sum_{i=1}^l b_ia_i$ and hence  $a$ is an element of the right-hand side of 
\eqref{center1} as well.
Therefore $\zz(L_R(E))_0$ is contained in the right-hand side and we have shown \eqref{center1}.}

(ii). We start the proof of this item by showing that the set in the {right-hand side} of 
\eqref{center2} is contained in $\zz(L_R(E))_n$.  Suppose $c$ is a cycle without exits such that $\vert F_E({c^0})\vert<\infty$.  Consider 
$$a:= \sum\limits_{\tiny
{\begin{matrix}
 d \in [c]\\
 \a \in F_E(c^0) \cup \{s(d)\}\\
 r(\a)=s(d)
 \end{matrix}
 }}\a d^m\a^\ast, \quad \hbox{where}\quad m\vert c \vert = n.$$
 {Then $a \in L_R(E)_n$.  Also} $$
f:=\pi(a)=\sum\limits_{\tiny
{\begin{matrix}
 d \in [c]\\
 \a \in F_E(c^0) \cup \{s(d)\}\\
 r(\a)=s(d)
 \end{matrix}
 }} 1_{Z(\alpha d^m, \a)}.$$
Notice that for each $d\in [c]$ and every $\a \in F_E({c^0}\cup \{s(d)\})$ such that 
$r(\a) =s(d)$ we have $Z(\a d^m, \a) = \{(\a d^\infty, n, \a d^\infty)\}$ because $c$ has no exits.

{To show that $a \in \zz(L_R(E))_n$, it suffices to show that $f$ is a class function
by \cite[Proposition 4.13]{Steinberg}.} Condition (1) 
is clear. For Condition (2), fix $x, z \in G_E$ such that $r(x)=s(x)=s(z)$.
We claim that  $x\in {\rm supp}\ f$ if and only if $zxz^{-1}\in {\rm supp}\ f$. 
To see this, notice that $f(x) \neq 0$ if and only if $x=(\a d^\infty, n, \a d^\infty)$ for some $d\in [c]$ and some $\a\in F_E(c^0)\cup \{s(d)\}$ with $r(\a)=s(d)$. This happens if and only if $zxz^{-1}= (\b d^\infty, n, \b d^\infty)\in {\rm supp}\ f$ for some $\b\in F_E(c^0)\cup \{s(d)\}$ with $r(\b)=s(d)$.

Thus $f$ is a class function.  {Now since the generators of the right-hand side of \eqref{center2} are in
the $R$-submodule $\zz(L_R(E))_n$, we have that the right-hand side of \eqref{center2} is contained in $\zz(L_R(E))_n$ 
also.}  
\medskip

In what follows we prove that any element in the $n$-component of $\zz(L_R(E))$ is as in the right hand side of \eqref{center2}. Let $a$ be a nonzero element in $\zz(L_R(E))_n$. Let $f=\pi(a)\in \zz(A_R(G_E))_n$. By \cite[Proposition 4.13]{Steinberg} the element $f$
 is a class function. By Lemma \ref{CINCO} there exists $\{b_i\}_{i=1}^l\subseteq R\setminus \{0\}$ such that $f=\sum_{i=1}^l b_i f_i$, where $f_i=1_{f^{-1}(b_i)}$ and each $f_i$ is a class function. 

To see that $a$ is as in the right hand side of \eqref{center2} it suffices to show that
for every $i$, $f_i$ is the image under $\pi$ of an element  in the right hand side of \eqref{center2}.  
Fix $i\in \{1, \dots, l\}$.  {Write $B_i:=f^{-1}(b_i)$.} We claim that there exists a finite 
subset $C\subseteq [{\mathcal C}_{ne}^f]$ such that for each $[c]\in C$ there exists $m_c\in \N$ such that $m_c\vert c \vert=n$ and 
\begin{equation}\label{extralarge}
B_i= {{\bigcup_{[c]\in C}\Big{(} \bigcup_{
\tiny
{\begin{matrix}
 d \in [c]\\
 \a \in F_E(c^0) \cup \{s(d)\}\\
 r(\a)=s(d)
 \end{matrix}
 }
}Z(\a d^{m_c}, \a)\Big{)}}}.
\end{equation}

Indeed, using Lemma \ref{lem.segundo}
we see that $B_i\subseteq \Delta$. 
Thus the elements of $B_i$ are of the form $(\a d^\infty, n, \a d^\infty)$ for some 
cycle without exists $c$, some $d\in [c]$ and some $\a\in F_E(c^0)\cup \{s(d)\}$ with $r(\a)=s(d)$. Thus  there exists $m_c\in \N$ such that 
$$\{(\a d^\infty, n, \a d^\infty)\}= Z(\a d^{m_c}, \a).$$
 Denote by $C$ the set of the classes of cycles appearing in the expression of the elements of $B_i$. For every $[c]\in  C$ and $d\in [c]$,  let $D_d\subseteq F_E({c^0})\cup \{s(d)\}$ be 
the set containing all the $\a $'s  such that $\a d^\infty$ appears in the expression of some element of $B_i$.

Then 
$$B_i = \bigcup_{[c]\in C}\Big{(}\bigcup_{\a\in D_d} Z(\a d^{m_c}, \a) \Big{)}.$$
Since $B_i$ is compact then $C$ and each $D_d$ is finite. 

Fix $[c]\in C$. To prove \eqref{extralarge} we show that

\[\bigcup_{d\in [c]} D_d= F_E({c^0})\cup \{s(d) \ \vert \ d \in [c]\}.\]
 Take  $\b\in F_E({c^0})\cup \{s(d) \ \vert \ d \in [c]\}$. 
 Pick $\a\in D_{d_1}$ for some $d_1\in [c]$ such that $D_{d_1}$ appears in the expression of $B_i$ and let
  $x=(\a d_1^\infty, n, \a d_1^\infty)\in {\rm supp} (f_i)$.  Then, there 
  exists $\gamma \in {\rm Path}(E)$ such that $\a d_1^\infty= \gamma d^\infty$. Consider $z=(\b d^\infty, k, \gamma d^\infty)$, where $k= \vert \beta \vert - \vert \gamma \vert$. Thus
 $$zxz^{-1}= (\beta d^\infty, n,\beta d^\infty) \in {\rm supp} (f_i)$$
because $f_i$ is a class function by Lemma \ref{CINCO}, so $\b \in D_d$ as needed.
 \end{proof}

\medskip

\begin{corollary}{Let $E$ be an arbitrary graph and $R$ a  commutative unital ring.  
Then \[\zz(L_R(E))_n \neq 0\]} for some $n\neq 0$ if and only there exists a cycle without exits $c$ such that $\overline{c^0}$ satisfies Condition {\rm (F)} and $B_{\overline{c^0}}=\emptyset$. Equivalently, $U_{\overline{c^0}, \emptyset}$ is compact.
\end{corollary}
\begin{proof}
It follows by (ii) in Theorem \ref{thm:center} and Lemma \ref{lem:cicloscentrales}. 
\end{proof}

\medskip

{In the remainder of the paper, we will refine Theorem \ref{thm:center} 
by giving a basis for the center of a Leavitt path algebra. Note that in the left-hand-side of \eqref{center1}, the 
``$\operatorname{span}$'' cannot be changed to a ``$\oplus$'' because the terms may not be linearly independent.}

{Recall that $\oe$ denotes the set of open invariant subsets of $\geo$.}

\begin{lemma}\label{lm:ccdCompactos}
Let $E$ be an arbitrary graph. Then, $\oe$ satisfies the descending chain
condition for compact {open invariant} sets. Moreover, every nonempty compact open invariant subset 
contains a minimal compact {open invariant} subset.
\end{lemma}
\begin{proof}
Let $U$ be a compact set in $\oe$. By \cite[Theorem 3.3]{CMMS}, $U=U_{H, S}$ for some $(H, S)\in \te$, 
and by Proposition \ref{prop:compacidad}, $S=B_H$ and $H$ satisfies Condition (F). In particular, $H$ is finite. 

Now, consider a decreasing sequence $U_{H_1,B_{H_1}}\supseteq U_{H_2,B_{H_2}}\supseteq \dots$. This implies, again by \cite[Theorem 3.3]{CMMS}, that  $H_1\supseteq H_2 \supseteq \dots$. Since $H_1$ is finite, then there exists $i\in \N$ such that $U_{H_i,B_{H_i}}= U_{H_{i+n},B_{H_{i+n}}}$ for all $n\in \N$.
This proves that $\oe$ satisfies the descending chain condition for compact sets.

The ``moreover part" follows immediately.

\end{proof}

\begin{lemma}\label{lem:descompMin}
Let $U$ be a compact subset of $\oe$. Then $U=\sqcup_{V\in \mathcal M}V$, where $\mathcal M$ is the finite set of all 
minimal compact open invariant subsets contained in $U$. 
\end{lemma}
\begin{proof}
Let $U=U_{H, B_H}$ be a compact { open invariant element of} $\oe$. By Lemma \ref{lm:ccdCompactos} the set $\mathcal M$  is nonempty. Moreover, 
it contains a finite number of elements as $H$ is finite. {We claim that  $\cup_{V\in \mathcal M} V$ is a disjoint union.}  
Indeed, for $V, V'\in \mathcal M$, since $V\cap V'$ is compact, the minimality of $V$ (and of $V'$) implies $V\cap V'=\emptyset$ or $V\cap V'=V=V'$.

{Finally, to see that $\cup_{V\in \mathcal M} V$ coincides with $U$, write 
\[U=  \big{(}\cup_{V\in \mathcal M} V\big{)} \sqcup U', \quad \text{ where } \quad U'=U\setminus \big{(}\cup_{V\in \mathcal M} V\big{)}.\] 
Note that $U'$ is compact (as it is a closed subset of a compact set) and
$U' \subseteq \oe$. By way of contradiction, suppose $U'$ is nonempty.  Then by Lemma \ref{lm:ccdCompactos}, $U'$ 
contains a minimal compact subset, say $U''$. Since $U''\in \mathcal M$, we get a contradiction.}
\end{proof}

\begin{remark}\label{rm:prod}
\rm
It is not difficult to see that for $U$ and $V$ compact sets in $\oe$. 
\begin{enumerate}[\rm (i)]
\item If  $U\cap V=\emptyset$ then $1_U1_V=0$.
\item $1_U1_U=1_U$.
%such that $U\cap V=\emptyset$. Then $UV=0$ or $UV=U=V$.
\end{enumerate}
\end{remark}

\begin{theorem}\label{thm:base}
Let $E$ be an arbitrary graph and $R$ a  commutative  unital ring. Then: 
\begin{enumerate}[\rm (i)]
\item The zero component of the center of $L_R(E)$ is a free $R$-module. One basis is given by: 
\begin{equation}\label{base1}
 \footnotesize{
\Big{\{}\sum_{v\in H}v + \sum_{\a\in F_E'(H)}\a\a^\ast + 
{\underset{
\tiny{\begin{matrix}
\a\in {\rm Path}(E)\\
r(\a) \in B_H
\end{matrix}}}{\sum}} ( \a\a^\ast - \sum_{e\in F_\a}\a e e^\ast \a^\ast) \ \vert \
(H, B_H)\in\tecm
 \Big{\}}.}
 \end{equation}
 \item When the $n$-component of the center of $L_R(E)$ is nonzero it is a free $R$-module and one basis is given by:
 \begin{equation}\label{base2}
  \footnotesize{
 \bigsqcup_{\tiny{
 \begin{matrix}
 [c]\ \in\ [{\mathcal C}_{ne}^f] \\ m \vert c \vert = n
 \end{matrix}}
 }\Big{\{} 
 \sum_{
\tiny
{\begin{matrix}
 d \in [c]\\
 \a \in F_E(c^0) \cup \{s(d)\}\\
 r(\a)=s(d)
 \end{matrix}
 }
 }\a d^m\a^\ast
\Big{\}}.}
\end{equation}
\end{enumerate}
\end{theorem}
\begin{proof}

(i) By Theorem \ref{thm:center} and Lemma \ref{lem:descompMin}, the set  given in \eqref{base1} generates ${\mathcal Z}(L_R(E))_0$. 
Use Remark \ref{rm:prod} to obtain that it is a linearly independent set.

 (ii) Again by Theorem \ref{thm:center}  we have that the set given in \eqref{base2} generates ${\mathcal Z}(L_R(E))_n$.
 To see that we have a linearly independent set, consider $\a d^m\a^\ast$ and $\b t^n\b^\ast$ as in \eqref{base2}. Note that 
$x:=(\a d^m\a^\ast)(\b t^n\b^\ast)$ is zero except when $\a=\b$ and $d=t$. In this case,
$x=\a d^{m+n}\a^\ast$. This implies that the elements in \eqref{base2} are linearly independent.
\end{proof}

%%%%%%%%%%%%%%%%%%%%%%%%%%%%%%
%%%%%%%%%%%%%%%%%%%%%%%%%%%%%%

\end{document}